\documentclass[11pt]{amsart}
\usepackage[margin=1.35in]{geometry}
\usepackage{amscd,amsmath,amsxtra,amsthm,amssymb,stmaryrd,xr,mathrsfs,mathtools,enumerate,commath, comment}
\usepackage{stmaryrd}
\usepackage{xcolor}
\usepackage{commath}
\usepackage{comment}
\usepackage{tikz-cd}
\usepackage{longtable} 
\usepackage{pdflscape} 
\usepackage{booktabs}
\usepackage{hyperref}
\definecolor{vegasgold}{rgb}{0.77, 0.7, 0.35}
\definecolor{darkgoldenrod}{rgb}{0.72, 0.53, 0.04}
\definecolor{gold(metallic)}{rgb}{0.83, 0.69, 0.22}
\hypersetup{
 colorlinks=true,
 linkcolor=darkgoldenrod,
 filecolor=brown,      
 urlcolor=gold(metallic),
 citecolor=darkgoldenrod,
 pdftitle={Topological Iwasawa invariants and Arithmetic Statistics},
 }

\usepackage[all,cmtip]{xy}

\DeclareFontFamily{U}{wncy}{}
\DeclareFontShape{U}{wncy}{m}{n}{<->wncyr10}{}
\DeclareSymbolFont{mcy}{U}{wncy}{m}{n}
\DeclareMathSymbol{\Sh}{\mathord}{mcy}{"58}
\usepackage[T2A,T1]{fontenc}
\usepackage[OT2,T1]{fontenc}

\newtheorem{theorem}{Theorem}[section]
\newtheorem*{mainthm}{Main Theorem}
\newtheorem{lemma}[theorem]{Lemma}
\newtheorem{question}[theorem]{Question}
\newtheorem{conjecture}[theorem]{Conjecture}

\newtheorem{proposition}[theorem]{Proposition}
\newtheorem{corollary}[theorem]{Corollary}
\newtheorem{definition}[theorem]{Definition}

\numberwithin{equation}{section}

\theoremstyle{remark}
\newtheorem{remark}[theorem]{Remark}

\newcommand{\ord}{\mathrm{ord}}

\newcommand{\Z}{\mathbb{Z}}
\newcommand{\p}{\mathfrak{p}}
\newcommand{\Q}{\mathbb{Q}}
\newcommand{\F}{\mathbb{F}}

\newcommand{\cL}{\mathcal{L}}

\newcommand{\cK}{\mathcal{K}}
\newcommand{\cO}{\mathcal{O}}

\newcommand{\op}[1]{\operatorname{#1}}

\newcommand\mtx[4] { \left( {\begin{array}{cc}
 #1 & #2 \\
 #3 & #4 \\
 \end{array} } \right)}

\begin{document}
\title[Topological Iwasawa invariants and Arithmetic Statistics]{Topological Iwasawa invariants and Arithmetic statistics}

\author[Cédric Dion]{Cédric Dion}
\address[C.~Dion]{D\'epartement de Math\'ematiques et de Statistique, Universit\'e Laval, Pavillion Alexandre-Vachon, 1045 Avenue de la M\'edecine, Qu\'ebec, QC, Canada G1V 0A6
}
\email{cedric.dion.1@ulaval.ca}

\author[Anwesh Ray]{Anwesh Ray}
\address[A.~Ray]{Department of Mathematics\\
University of British Columbia\\
Vancouver BC, Canada V6T 1Z2}
\email{anweshray@math.ubc.ca}

\begin{abstract}
Given a prime number $p$, we study topological analogues of Iwasawa invariants associated to $\Z_p$-covers of the $3$-sphere that are branched along a link. We prove explicit criteria to detect these Iwasawa invariants, and apply them to the study of links consisting of $2$ component knots. Fixing the prime $p$, we prove statistical results for the average behaviour of $p$-primary Iwasawa invariants for $2$-bridge links that are in Schubert normal form. Our main result, which is entirely unconditional, shows that the density of $2$-bridge links for which the $\mu$-invariant vanishes, and the $\lambda$-invariant is equal to $1$, is $(1-\frac{1}{p})$. We also conjecture that the density of $2$-bridge links for which the $\mu$-invariant vanishes is $1$, and this is significantly backed by computational evidence. Our results are proven in a topological setting, yet have arithmetic significance, as we set out new directions in arithmetic statistics and arithmetic topology. 
\end{abstract}


\maketitle
 
\section{Introduction}
\par Let $p$ be a fixed odd prime number and $K$ be a number field. Iwasawa studied growth patterns of $p$-primary class groups in certain towers of number field extensions of $K$. More precisely, a $\Z_p$-extension of $K$ is a Galois extension $K_\infty$ of $K$ with Galois group $\op{Gal}(K_\infty/K)\simeq \Z_p$. For any positive integer $n$, let $K_n\subset K_\infty$ be the unique extension of $K$ with $[K_n:K]$, and let $\op{Cl}_p(K_n)$ be the $p$-Sylow subgroup of the class group of $K_n$. Iwasawa \cite{iwasawa1973zl} showed that for $n\gg0$,
\begin{equation}\label{Iwasawa first formula}\# \op{Cl}_p(K_n)=p^{p^n \mu+n \lambda+\nu},\end{equation} where $\mu, \lambda\in \Z_{\geq 0}$ and $\nu\in \Z$ are Iwasawa invariants associated to the $\Z_p$-extension $K_\infty/K$. The characterization of Iwasawa invariants is the subject of much mystery and speculation. Letting $K_{\op{cyc}}/K$ be the \emph{cyclotomic $\Z_p$-extension}, the associated $\mu$ and $\lambda$ Iwasawa invariants shall (in this section) be denoted $\mu_p(K)$ and $\lambda_p(K)$. 
\par Our primary motivation is to study the statistical distribution of such invariants. Given a fixed prime number $p$, one is interested in the average behavior of $\mu_p(K)$ and $\lambda_p(K)$. More precisely, we fix $n>0$, and let $\mathcal{F}$ be an infinite family of number fields $K$ with $[K:\Q]=n$, ordered by absolute discriminant $|\Delta_K|$. For instance, when $n=2$, it is natural to consider the family of all imaginary quadratic fields, or the family of all real quadratic fields. Furthermore, it is natural in this context to impose splitting conditions on $p$. Consider the following question.
\begin{question}\label{main question}
Fix a prime number $p$ and $\mu, \lambda\in \Z_{\geq 0}$.
Given a family of number fields $\mathcal{F}$ of degree $n$, let $\mathcal{F}^{\mu, \lambda}$ be the subset of $\mathcal{F}$ consisting of number fields $K$ such that $\mu_p(K)=\mu$ and $\lambda_p(K)=\lambda$. Given a real number $x>0$, let $\mathcal{F}_{<x}$ be the subset of $\mathcal{F}$ of all $K$ such that $|\Delta_K|<x$ and set $\mathcal{F}^{\mu, \lambda}_{<x}:=\mathcal{F}^{\mu, \lambda}\cap \mathcal{F}_{<x}$. Then, what can be said about the proportion of number fields $K$ in $\mathcal{F}$ with $\mu_p(K)=\mu$ and $\lambda_p(K)=\lambda$, i.e., the limit
\begin{equation}\label{first equation}\mathfrak{d}(\mathcal{F}^{\mu, \lambda})=\lim_{x\rightarrow \infty} \frac{\# \mathcal{F}^{\mu, \lambda}_{<x}}{\# \mathcal{F}_{<x}}?\end{equation}
\end{question}
Note that the limit \eqref{first equation} need not exist, however, one may ask the same questions for the upper and lower limits, defined by replacing the limit by $\limsup$ and $\liminf$ respectively. Even for the simplest families of number fields, i.e. imaginary/real quadratic fields, such questions have proved difficult. On the other hand, the systematic study of arithmetic statistics for the Iwasawa invariants of elliptic curves was initiated by the second named author, and has gained considerable momentum in the following works: \cite{ray2021statistics1, ray2021statistics2, ray2021statistics3, ray2021arithmetic4, ray2021rank, ray2021arithmetic5,ray2021arithmetic6}.

\par The main goal of the present work is to study a topological analogue of question \ref{main question}, which has proven to be a more tractable problem. There is a deep analogy called \emph{arithmetic topology} between number fields and $3$-manifolds. This analogy asserts a mysterious parallel between primes in number fields and knots embedded in a $3$-manifold. As a result, there are many arithmetic phenomena for number fields that are reflected on the topological side and vice versa. This philosophy originates from observations made by Mazur in \cite{mazur1963remarks}, which he attributes to conversations with Mumford. For a modern account of various aspects of this analogy, we refer to \cite{morishita2011knots}. Guided by this analogy, Hillman, Matei and Morishita \cite{hillman2005pro} have defined analogues of Iwasawa invariants associated to a link $\cL$ embedded in a $3$-sphere. Such Iwasawa invariants are used to prove an asymptotic formula for the growth patterns of $p$-homology in a $\Z_p$-cover of $S^3$ that is branched along $\cL$. We refer to \cite[Theorem 5.1.7]{hillman2005pro} for the precise statement. This formula has a strong resemblance to the classical formula \eqref{Iwasawa first formula} due to Iwasawa. It is not difficult to prove that for any knot, the associated $\mu$ and $\lambda$ invariants vanish (see Lemma \ref{mu lambda for a knot}). Therefore, the study of Iwasawa invariants is of genuine interest only for links with at least $2$ knot components.
\par In this paper, we consider the family of links in $S^3$ with $2$ knot components. For such links, the $\lambda$-invariant is $\geq 1$. Such links are presented as \emph{2-bridge links}, and there is an explicit parametrization called the \emph{Schubert form}, which associates a rational number $a/b$ with $(a,b)=1$, $b$ even and $0<a<2b$ to an explicit presentation which is denoted $\cL_{a/b}$. With respect to counting function $\op{ht}(\cL_{a/b}):=\max\{a^2,b\}$, our main result shows that for a density of $(1-1/p)$ links, the $\mu$-invariant vanishes and the $\lambda$-invariant is equal to $1$. In greater detail, an integral vector $z=(z_1, z_2)$ is called admissible if $z_1 z_2\neq 0$ and $\op{gcd}(z_1, z_2)=1$. The choice of $z$ fixes a $\Z_p$-cover of the link complement $S^3\backslash \cL_{a/b}$. Associated with this cover are $\mu$ and $\lambda$-invariants, $\mu_{p, a/b,z}$ and $\lambda_{p,a/b,z}$. Let $\mathcal{N}$ be the set of all fractions $a/b$ with $(a,b)=1$, $b$ even, $0<a<2b$. For $x>0$, set $\mathcal{N}_{<x}$ to be the subset of $\mathcal{N}$ consisting of fractions $a/b$ such that $\op{ht}(\cL_{a/b})<x$. It is easy to see that $\mathcal{N}_{<x}$ is finite, we set $\#\mathcal{N}_{<x}$ to denote its cardinality. Given a subset $\mathcal{S}$ of $\mathcal{N}$, let $\mathcal{S}_{<x}:=\mathcal{S}\cap \mathcal{N}_{<x}$. We say that $\mathcal{S}$ has density $\mathfrak{d}(\mathcal{S})$ if the following limit exists
\[\mathfrak{d}(\mathcal{S}):=\lim_{x\rightarrow \infty} \frac{\# \mathcal{S}_{<x}}{\# \mathcal{N}_{<x}}.\] Denote by $\mathcal{S}^{p,z}$ the set of all fractions $a/b$ in $\mathcal{N}$ such that $\mu_{p,\cL_{a/b},z}=0$ and $\lambda_{p,\cL_{a/b},z}=1$. Our main result is as follows. 
\begin{mainthm}[Theorem \ref{main th}]
With respect to notation above, $\mathcal{S}^{p,z}$ has density given by $\mathfrak{d}(\mathcal{S}^{p,z})=1-1/p$.
\end{mainthm}
\par We stress that the above result entirely unconditional. Furthermore, in section \ref{s 6}, we provide computational evidence for the vanishing of the $\mu$-invariant for $100\%$ of links parametrized in this way, for $z=(1,1)$.

\par \emph{Organization:} In section \ref{s 2}, we outline the parallels between topology and arithmetic and make preliminary observations. Section \ref{s 3} is devoted to the study of the \emph{Alexander polynomial} as well as certain $p$-adic analogs. In section \ref{s 4}, we prove explicit criteria that relate \emph{linking numbers} to Iwasawa invariants. We state and prove our main results in section \ref{s 5} for the family of $2$-bridge links. Section \ref{s 6} provides examples of links with positive $\mu$-invariant (for the total linking number covering space), and computational data showing that one may expect that $\mu=0$ for $100\%$ (i.e., a density $1$ set) of $2$-bridge links.
\subsection{Acknowledgments}
The first named author's research is supported by the Canada Graduate Scholarships – Doctoral program from the Natural Sciences and Engineering Research Council of Canada. The authors would like to thank Antonio Lei for helpful comments.
\section{Parallels between topology and arithmetic}\label{s 2}
In this section, we introduce some preliminary notions and delineate the parallels between topology and arithmetic. This shall motivate our investigations in \emph{topological Iwasawa theory}, i.e., the Iwasawa theory of the homology of various covers of $S^3$ (the $3$-dimensional sphere), that are branched along a fixed link.

\par Before discussing the topological analog, we recall the classical Iwasawa theory of a number field. For a comprehensive account, the reader is referred to the original work of Iwasawa \cite{iwasawa1973zl}. Fix a prime number $p$ and a number field $K$. Let $\bar{K}$ be a choice of algebraic closure of $K$ and $\op{G}_K:=\op{Gal}(\bar{K}/K)$ the \emph{absolute Galois group} of $K$. Given a finite set of prime numbers $S$, let $K_S$ be the maximal extension of $K$ contained in $\bar{K}$ in which all primes $v\notin S$ are unramified. Set $\op{G}_{K,S}$ to denote the Galois group $\op{Gal}(K_S/K)$. A $\Z_p$-extension is a Galois extension $K_\infty/K$ such that $\op{Gal}(K_\infty/K)$ is isomorphic to $\Z_p$ as a topological group. Note that all primes $v\nmid p$ of $K$ are unramified in $K_\infty$. The cyclotomic $\Z_p$-extension $K_{\op{cyc}}/K$ is the unique $\Z_p$-extension contained in $K(\mu_{p^\infty})$. Given a $\Z_p$-extension $K_\infty$, let $K_n$ be the unique extension of $K$ which is contained in $K_\infty$, such that $[K_n:K]=p^n$. Denote by $\op{Cl}_p(K_n)$ the $p$-Sylow subgroup of the class group of $K_n$. Iwasawa proved the following asymptotic formula for the growth of the $p$-class groups
\[\# \op{Cl}_p(K_n)=p^{p^n \mu+n \lambda+\nu}\]for $n\gg 0$, where $\mu, \lambda\in \Z_{\geq 0}$, and $\nu\in \Z$ are Iwasawa invariants associated to $K_\infty/K$. Thus in particular, $\mu=\lambda=0$ if and only if $\#\op{Cl}_p(K_n)$ is bounded in $n$. For the cyclotomic $\Z_p$-extension $K_{\op{cyc}}/K$, it is conjectured by Iwasawa that the $\mu$-invariant of $K_{\op{cyc}}/K$ must vanish. This was proven by Ferrero and Washington \cite{ferrero1979iwasawa} in the specific case when $K$ is an abelian extension of $\Q$.
\par Given a $\Z_p$-extension $K_\infty/K$, set $X\left(K_\infty\right)$ to be the inverse limit of the $p$-primary class groups $\op{Cl}_p(K_n)$ with respect to norm maps. Note that $X\left(K_\infty\right)$ is a $\Z_p$-module as well as a module over $\Gamma=\op{Gal}(K_\infty/K)$. The Iwasawa algebra associated to $\Gamma$ is defined to be the inverse limit of finite group algebras $\widehat{\Lambda}:=\varprojlim_n \Z_p[\Gamma/\Gamma^{p^n}]$. Letting $\gamma$ be a topological generator of $\Gamma$, there is an isomorphism $\Z_p\llbracket X\rrbracket \simeq \widehat{\Lambda}$ which identifies the formal variable $X$ with $(\gamma-1)$. We note here that in most expositions in Iwasawa theory, it is customary to let $\Lambda$ denote the Iwasawa algebra, however, we follow the notation used in \cite{hillman2005pro} and \cite{morishita2011knots} in which $\Lambda$ is reserved for the Laurent series ring $\Z[t^{\pm1}]$.
\par Given a finitely generated and torsion $\widehat{\Lambda}$-module $M$, the structure theorem \cite[Theorem 13.12]{lawrence1997washington} states that there is a map of $\widehat{\Lambda}$-modules
\[
M^{\vee}\longrightarrow \left(\bigoplus_{i=1}^s \widehat{\Lambda}/(p^{\mu_i})\right)\oplus \left(\bigoplus_{j=1}^t \widehat{\Lambda}/(f_j(X)) \right)
\]
with finite kernel and cokernel. Here, $\mu_i>0$ and $f_j(X)$ is a distinguished polynomial (i.e., a monic polynomial with non-leading coefficients divisible by $p$).
The characteristic ideal of $M$ is defined to be up to a unit generated by
\[
 f_{M}(X) := p^{\sum_{i} \mu_i} \prod_j f_j(X)=p^\mu g_M(X),
\]
where $\mu=\sum_i \mu_i$ and $g_M(X)$ is the distinguished polynomial $\prod_j f_j(X)$. The quantity $\mu$ is the $\mu$-invariant, which is denoted $\mu_p(M)$, and the $\lambda$-invariant $\lambda_p(M)$ is the degree of $g_M(T)$. For ease of notation, we simply set $\mu_p(K_\infty/K):=\mu_p\left(X(K_\infty)\right)$ and $\lambda_p(K_\infty/K):=\lambda_p\left(X(K_\infty)\right)$. We let $K_{\op{cyc}}$ be the cyclotomic $\Z_p$-extension of $K$. Thus, the result of Ferrero and Washington states that if $K/\Q$ is abelian and $p$ is any prime, then, $\mu_p(K_\infty/K)=0$. We shall simply refer to the characteristic polynomial of $X(K_\infty)$ as the $p$-primary Iwasawa polynomial of $K_\infty/K$.

\par Our main focus will be to study the behavior of various topological analogues of such Iwasawa invariants that arise naturally in the context of knot theory. First, we explain the parallels between the Galois theory of number fields with restricted ramification, and the topological theory of coverings of $S^3$ that are branched along a link. Let $K$ be a number field and $\cO_K$ be the ring of integers. 

\par Class field theory is based on certain duality properties of the $1$-dimensional scheme $X=\op{Spec} \cO_K$. Let $\mathbb{G}_{m, X}$ be the \'etale sheaf on $X$, which associates to an \'etale cover $\varphi:\op{Spec} B\rightarrow X$, the group of multiplicative units $\mathbb{G}_{m, X}(\varphi)=B^{\times}$. Given a \emph{constructible sheaf} $\mathfrak{M}$ on $X$ (see \cite[p.146]{milne2006arithmetic} for the definition), set $\mathfrak{M}^\vee:=\underline{\op{Hom}}(\mathfrak{M}, \mathbb{G}_{m, X})$ (where $\underline{\op{Hom}}(\cdot, \cdot)$ is standard notation for the sheaf of homomorphisms, see \cite[Chapter 2]{hartshorne2013algebraic} or \emph{loc. cit.}). We denote by $\hat{H}^i(X, \mathfrak{M})$ the modified \'etale cohomology groups of $\mathfrak{M}$. The Artin-Verdier duality theorem (see \cite[Theorem 1.8]{milne2006arithmetic}) states that there is a canonical isomorphism $\hat{H}^3(X, \mathbb{G}_{m, X})\simeq \Q/\Z$ and a nondegenerate pairing of abelian groups
\[\hat{H}^i(X, \mathfrak{M}^\vee)\times \op{Ext}_X^{3-i}(\mathfrak{M}, \mathbb{G}_{m, X})\rightarrow \hat{H}^3(X, \mathbb{G}_{m, X})\xrightarrow{\sim} \Q/\Z.\]
The above pairing is analogous to the Poincare duality pairing for a compact oriented $3$-manifold $Y$ and foreshadows a deep resemblance between the arithmetic properties of a number ring $\cO_K$ and the topological properties of an oriented $3$-manifold $Y$.  

\par There is also a close analogy between embedded knots in $3$-manifolds and primes in number rings, which we briefly explain. For further details, the reader may refer to \cite[chapter 3]{morishita2011knots}. Given a finite field $\F$, the \'etale fundamental group of $\op{Spec}\F$ is given by \[\pi_1(\op{Spec} \F)=\op{Gal}(\bar{\F}/\F)\simeq \hat{\Z}.\] In fact, $\op{Spec} \F$ is the \'etale analogue of the Eilenberg-Maclane space $K(\widehat{\Z}, 1)$ (see \cite[p.87, p.365]{hatcheralgebraic} for the definition of this notion). On the topological side, since $\pi_1(S^1)\simeq \Z$ and $\pi_i(S^1)=0$ for $i\geq 2$, the circle $S^1$ is $K(\Z, 1)$. For every integer $n\geq 1$, there is a unique extension $\F'/\F$ of degree $n$ with $\op{Gal}(\F'/\F)\simeq \Z/n\Z$. On the topological side, there is a unique $\Z/n\Z$ cover $[n]:S^1\rightarrow S^1$. Identifying $S^1$ with the unit circle in $\mathbb{C}$, the covering map $[n]$ maps $z\mapsto z^n$. Thus the finite Galois covers of $\F$ are in correspondence with finite covers of $S^1$. 
\par Given a prime ideal $\p$ in a number ring $\cO_K$, let $\F_{\p}$ be the residue field at $\p$. Let $K_{\p}$ be the completion of $K$ at $\p$, and $\cO_{\p}$ the valuation ring in $K$. In the diagram
\begin{equation}\label{num field diagram}\op{Spec} K_{\p} \hookrightarrow \op{Spec} \cO_{\p}\hookleftarrow \op{Spec} \F_{\p},\end{equation} the map $\op{Spec}\F_{\p}\hookrightarrow \op{Spec} \cO_{\p}$ is induced by reduction mod-$\mathfrak{p}$ and is an \emph{\'etale homotopy equivalence} (see \cite{artin2006etale}). On the other hand, the map $\op{Spec} K_{\p} \hookrightarrow \op{Spec} \cO_{\p}$ is induced by the inclusion of $\cO_{\p}$ in $K_{\p}$.
\par On the topological side, the ambient $3$-manifold is assumed for simplicity to be $S^3$, and there would be significant work required to extend the results to links embedded in more general $3$-manifolds. We note that some authors do work in slightly greater generality, for instance, in \cite{ueki2017iwasawa}, results are proven for \emph{rational homology $3$-spheres}. We come to the definition of a \emph{knot} and more generally, a \emph{link}. We refer to \cite[Chapter 1]{kawauchi1996survey} and \cite[Chapter 1]{lickorish2012introduction} for a comprehensive treatment of some of the basic notions in knot theory, as well as definitions that are equivalent to that given below, which we have chosen since in our estimation, it is easiest to state. 

\begin{definition}
A knot (in $S^3$) is an embedding $\iota: S^1\hookrightarrow S^3$ which is locally flat, i.e., locally homeomorphic to the canonical embedding $\mathbb{R}\hookrightarrow \mathbb{R}^3$ sending $x$ to $(x,0,0)$. Two knots are equivalent if there is an isotopy $f_t: S^3\rightarrow S^3$, $0\leq t\leq 1$, taking one knot to the other. In order to simplify notation, we shall identify the embedding $\iota:S^1\hookrightarrow S^3$ with the $\cK:=\iota(S^1)$. A \emph{link} (in $S^3$) is a union of disjoint knots $\mathcal{L}=\cK_1\cup \dots \cup \cK_r$ in $S^3$.
\end{definition}
Let $\cK$ be a knot in $S^3$, fix a \emph{tubular neighbourhood} (see \cite[p.137, ll.7-11]{lee2000introduction} for the definition) $\mathcal{V}$ of $\cK$, and let $\partial \mathcal{V}$ be the boundary of $\mathcal{V}$. Note that $\partial\mathcal{V}$ is homotopic to $S^1\times \cK$. The diagram \eqref{num field diagram} is analogous to 
\[\partial \mathcal{V} \hookrightarrow \mathcal{V}\hookleftarrow \mathcal{K}.\]

Note that the inclusion $\mathcal{K}\hookrightarrow \mathcal{V}$ is a homotopy equivalence, and $\partial \mathcal{V}$ is homeomorphic to $S^1\times \mathcal{K}$. By abuse of notation, we identify $\mathcal{V}$ with the product $S^1\times \mathcal{K}$. Consider the composite of maps \[\xi: \pi_1(\partial \mathcal{V})\rightarrow \pi_1(\mathcal{V})\xrightarrow{\sim} \pi_1(\mathcal{K}).\] Given $a\in S^1$, we refer to a loop of the form $\{a\}\times \mathcal{K}$ in $\partial \mathcal{V}$ as a \emph{longitude} of $\mathcal{K}$, and we denote the associated homotopy class by $\ell_{\mathcal{K}}\in \pi_1(\partial \mathcal{V})$. Note that this homotopy class is independent of the choice of $a$. The image of $\ell_{\mathcal{K}}$ is a generator of $\pi_1(\mathcal{K})$, and for this reason, $\ell_{\mathcal{K}}$ is an analog of a Frobenius element $\sigma_{\p}\in \op{Gal}(\bar{K}_{\p}/K_{\p})$. The kernel of $\xi$ is referred to as the \emph{inertia group} of $\cK$. A \emph{meridian} $\tau_{\mathcal{K}}$ is a generator of the inertia group of the form $S^1\times \{b\}$. Note that this homotopy class is independent of the choice of point $b$. The fundamental group $\pi_1(\partial \mathcal{V})$ is the generated by $\ell_{\cK}$ and $\tau_{\cK}$, with the single relation $[\ell_{\cK}, \tau_{\cK}]=\ell_{\cK}\tau_{\cK}\ell_{\cK}^{-1} \tau_{\cK}^{-1}=1$.
\par Let $\mathcal{L}=\cK_1\cup \dots \cup \cK_r$ be an \emph{oriented link}, which is to say that each knot $\mathcal{K}_j$ is oriented. Given distinct knots $\mathcal{K}_i$ and $\mathcal{K}_j$, let $\ell_{i,j}$ be the \emph{linking number} of $\mathcal{K}_i$ and $\mathcal{K}_j$, we refer to \cite[p. 11]{kawauchi1996survey} for a precise definition. Let $X_{\cL}$ be the complement of $\cL$ in $S^3$, and let $x_0$ be a choice of base-point in $X_{\cL}$.
\begin{definition}
The \emph{link group} of $\cL$ is the fundamental group $G_{\cL}:=\pi_1(X_{\cL};x_0)$.
\end{definition} A comprehensive study of the analogies between fundamental groups and Galois groups of number fields is undertaken in \cite{volklein1996groups}. The \emph{link group} is the analogue of the group $\op{Gal}(K_S/K)$, where $S$ is a finite set of primes and $K_S$ is the maximal algebraic extension of $K$ which is unramified at primes $v\notin S$. The set of primes $S$ here is the set at which $K_S$ may ramify and likewise. On the other hand, $G_{\cL}$ is the deck group of $\pi_{\cL}:\widetilde{X}_{\cL}\rightarrow X_{\cL}$, the universal cover of $X_{\cL}$.

\section{The Alexander Module associated to a link in a $3$-sphere}\label{s 3}

\par The classical Alexander polynomial is an important topological invariant associated to a link in a $3$-sphere (or, more generally $3$-manifold). It was first introduced by Alexander \cite{alexander1928topological}. On the other hand, there is a $p$-adic version introduced by Hillman-Matei-Morishita \cite{hillman2005pro} which is widely regarded as the multi-variable topological analogue of the Iwasawa polynomial of $\Z_p$-extension. In this setting, one may define $\mu$, $\lambda$ and $\nu$ invariants, and these invariants give a formula for the asymptotic growth of the $p$-parts in the homology for the various finite covering spaces occuring in a chosen $\Z_p$-cover which is branched along $\cL$. In this section, we recall the notion of the completed Alexander polynomial and associated Iwasawa invariants, and explain the relationship with $p$-homology classes occurring in a $\Z_p$-cover.
\par Let $\cL$ be a link in $S^3$ consisting of component knots $\cK_1,\dots,\cK_r$, and denote the complement by $X_{\cL}:=S^3\backslash \cL$. For each knot $\mathcal{K}_i$, we let $x_i\in G_{\cL}$ the class generated by the meridian looping about $\cK_i$ starting and ending at the base point $x_0$. Note that this involves a choice of a small enough tubular neighbourhood of $\cK_i$, that is contained in $X_{\cL}$. We set $\pi_{\cL}^{\op{ab}}: \widetilde{X_{\cL}}^{\op{ab}}\rightarrow X_{\cL}$ to be the maximal abelian cover of $X_{\cL}$ and $G_{\cL}^{\op{ab}}$ the maximal abelian quotient of $G_{\cL}$. Then note that there is a natural isomorphism \[G_{\cL}^{\op{ab}}\simeq \op{Gal}\left(\widetilde{X_{\cL}}^{\op{ab}}/X_{\cL}\right).\]By the Hurewicz theorem, $G_{\cL}^{\op{ab}}$ is isomorphic to $H_1(X_{\cL};\Z)$. The homology group $H_1(X_{\cL};\Z)$ is generated by $t_i=x_i[G_{\cL},G_{\cL}]$, $i=1,\dots, r$, as a free $\Z$-module of rank $r$. Letting $\Lambda_r$ be the group algebra $\Z[G_{\cL}^{\op{ab}}]$, we find that $\Lambda_r$ is the Laurent polynomial ring
\[\Lambda_r=\Z[t_1^{\pm 1}, \dots, t_r^{\pm 1}].\] \begin{definition}\label{def alexander module}
Let $\cL$ be a link with $r$ components as above. The \emph{Alexander module} is the relative homology group $A_{\cL}=H_1(\widetilde{X_{\cL}}^{\op{ab}}, \mathcal{F}_{x_0}; \Z)$, where $\mathcal{F}_{x_0}$ is the fibre of $x_0$ in $\widetilde{X_{\cL}}^{\op{ab}}$. 
\end{definition}

Note that the Galois group $\op{G}_{\cL}^{\op{ab}}$ acts on $\widetilde{X_{\cL}}^{\op{ab}}$ by deck transformations, and this gives rise to action of $\op{G}_{\cL}^{\op{ab}}$ on $A_{\cL}$. As a result, the Alexander module and link modules are modules over $\Lambda_r$. In order to define the Alexander polynomial, we will first recall some prerequisite notions from commutative algebra.
\par Let $R$ be a noetherian unique factorization domain and $M$ a finitely generated $R$-module. Let $u_1,\dots, u_q$ be a finite set of generators of $M$ as an $R$-module, and $d_1:R^q\rightarrow M$ the surjective homomorphism sending the $i$-th coordinate generator $e_i=(0,\dots, 0, 1,0,\dots)$ to $u_i$. Let $d_2:R^\ell\rightarrow R^q$ be an $R$-module homomorphism such that $R^\ell$ surjects onto the kernel of $d_1$. Note that the maps fit into a right exact sequence 
\[R^\ell\xrightarrow{d_2} R^q\xrightarrow{d_1} M\rightarrow{0}.\] The map $d_2$ is represented by a matrix $A$. The $i$-th elementary ideal $E^{(i)}_R(M)$ (also called the $i$-th Fitting ideal) is the ideal generated by all determinants of $(q-i)\times (q-i)$ submatrices in $A$. The \emph{divisorial hull} of any ideal $I$ in $R$ is the intersection of all principal ideals containing $I$. We let $\Delta_R^{(i)}(M)$ be the divisorial hull of $E_R^{(i)}(M)$.
\par The maximal abelian extension $\pi_{\cL}^{\op{ab}}:\widetilde{X_{\cL}}^{\op{ab}}\rightarrow X_{\cL}$ contains various subextensions \[\pi_{\cL,z}:X_{\cL,z}\rightarrow X_{\cL}\] such that $\op{Gal}(X_{\cL, z}/X_{\cL})\simeq \Z$. Here, $z$ is an a choice of integral vector $z=(z_1,\dots, z_r)\in \Z^r$ such that 
\begin{itemize}
    \item $\op{gcd}(z_1,\dots, z_r)=1$,
    \item and $\prod_i z_i\neq 0$. 
\end{itemize} 
An integral vector satisfying the above conditions is said to be \emph{admissible}. Given an admissible $z$, define a surjective homomorphism $\sigma_z: H_1(X_{\cL};\Z)\rightarrow \Z$ sending $t_i$ to $z_i$ for $i=1,\dots , r$. Let $H_z$ be the kernel of $\sigma_z$. Note that there corresponds to $H_z$ a cover \[\pi_{\cL,z}:X_{\cL,z}\rightarrow X_{\cL}\] whose deck group is canonically identified with $\op{G}_{\cL}/H_z\simeq \Z$. This gives a family of $\Z$-covers, and this family is infinite if $r\geq 2$. Letting $\mathbf{1}=(1,\dots, 1)$, we call $X_{\cL,\mathbf{1}}$ the \emph{total linking number} covering space of $X_{\cL}$. For ease of notation we abbreviate $X_{\cL,z}$ to $X_z$, and $\pi_{\cL,z}$ to $\pi_z$.
\begin{definition}
The \emph{multivariable Alexander polynomial} associated to $\cL$ is defined as
\[\Delta_{\cL}(t_1,\dots, t_r)=\Delta_{\Lambda_r}^{(1)}(A_{\cL}).\]Let $z$ be an admissible integral vector. The \emph{reduced Alexander polynomial} associated to $X_z$ is defined in $\Lambda_1\simeq \Z[\op{Aut}(X_z/X)]$ as 
\[\Delta_{L,z}(t)=\Delta_{\Lambda_1}^{(0)}\left(H_1(X_z;\Z)\right)=\Delta_{\cL} (t^{z_1},\dots, t^{z_r}).\]
\end{definition}

Following \cite{hillman2005pro}, we now define the completed Alexander module. The prime $p$ is fixed and will be suppressed in much of our notation. We let $\widehat{G_{\cL}}$ be the pro-$p$ completion of $G_{\cL}$ defined as the inverse limit
\[\widehat{G_{\cL}}:=\varprojlim_N G_{\cL}/N\]where $N$ ranges over all normal subgroups of $G_{\cL}$ with finite $p$-power index. We enumerate the knot components in $\cL$ by $\cK_1,\dots, \cK_r$ and let $x_i,y_i\in \widehat{G_{\cL}}$ be the elements corresponding to the meridian and longitude around $\cK_i$ respectively. Then Milnor shows in \cite{milnor1957isotopy} that there is a presentation of the form
\[\widehat{G_{\cL}}=\langle x_1,\dots, x_r\mid [x_i,y_i]=1 \text{ for }i=1,\dots, r \rangle. \]Given a pro-$p$ group $\mathcal{G}$, the associated $p$-adic Iwasawa algebra is the inverse limit 
\[\widehat{\Lambda}(\mathcal{G}):=\varprojlim_\mathcal{N} \Z_p[\mathcal{G}/\mathcal{N}],\] where $\mathcal{N}$ ranges over all finite index normal subgroups of $\mathcal{G}$. Set $\widehat{H}:=\Z_p^r$, the associated Iwasawa algebra is denoted $\widehat{\Lambda}_r$, and it is easy to show that $\widehat{\Lambda}_r$ is isomorphic to the formal power series ring $\Z_p\llbracket X_1,\dots, X_r\rrbracket$. Here, we choose a basis $\gamma_1, \dots \gamma_r$ of $\widehat{H}$ and let $X_i$ be the formal variable which corresponds to $(\gamma_i-1)$. Consider the quotient $\widehat{G_{\cL}}\rightarrow \widehat{H}$ given by sending $x_i$ to $\gamma_i$, which in turn shows that there is a quotient $\Lambda(G_{\cL})\rightarrow \widehat{\Lambda}_r$ mapping $x_i$ to $\gamma_i$. This map is well defined since the relations $[x_i,y_i]$ must map to $0$ as $\widehat{H}$ is abelian. 
\par We let $F$ be the (non-abelian) free group in $r$-generators $x_1,\dots, x_r$ and $\widehat{F}$ its pro-$p$ completion. Denote by $\epsilon: \Lambda(\widehat{F})\rightarrow \Z_p$ the augmentation map, sending every group-like element to $1$. It follows from the work of Ihara \cite{ihara1986galois} that the Fox differential calculus extends to $\widehat{F}$. One may define $\Z_p$-homomorphisms for $i=1,\dots, r$,
\[\partial_i=\frac{\partial}{\partial x_i}: \Lambda(\widehat{F})\rightarrow \Lambda(\widehat{F})\] such that any element $\alpha\in \Lambda(\widehat{F})$ can be uniquely expressed as a sum 
\[\alpha=\epsilon(\alpha)+\sum_{i=1}^{r}\partial_i(\alpha) (x_i-1).\] The derivatives of higher order are defined inductively as follows. Given a sequence $I=(i_1,\dots, i_m)$ with $1\leq i_j\leq r$, set 
\[\partial_I:=\partial_{i_1}\dots \partial_{i_m} (\alpha).\] We refer to \cite[section 2]{hillman2005pro} for the basic properties of $\partial_I$.
\par Denote by $\Z_p\langle X_1,\dots, X_r\rangle $ the formal non-commuting power series in $\Z_p$. An element of $\Z_p\langle X_1,\dots, X_r\rangle $ is represented as a formal sum of monomials $X_I=X_{i_1}X_{i_2}\dots X_{i_n}$, where $I=(i_1,\dots ,i_n)$. For $m\geq 1$, let $I(m)$ be the ideal in  $\Z_p\langle X_1,\dots, X_r\rangle $ which is  generated by all monomials $X_{i_1}\dots X_{i_m}$. Let $\Z_p\langle \langle X_1,\dots, X_r\rangle \rangle$ be the formal power series ring in $n$ noncommuting variables, defined as the completion
\[\Z_p\langle \langle X_1,\dots, X_r\rangle \rangle:=\varprojlim_m \left(\frac{\Z_p\langle X_1,\dots, X_r \rangle}{I(m)}\right).\]The pro-$p$ Magnus embedding $M:\widehat{F}\rightarrow \Lambda(\widehat{F})$ is defined by
\[M(x_i)=1+X_i\text{ and }M(x_i^{-1})=\sum_{j=0}^\infty (-1)^j X_i^j.\]
\par Denote by $\hat{\psi}:\Lambda(\widehat{G}_{\cL})\rightarrow \widehat{\Lambda}_r=\Z_p\llbracket X_1,\dots, X_r\rrbracket$ the natural quotient map. Following \cite[section 3]{hillman2005pro}, the Alexander matrix is the $r\times r$ matrix $\widehat{P}_{\cL}:=\left(\hat{\psi}(\partial_j[x_i,y_i])\right)$ and the completed Alexander module over $\widehat{\Lambda}_r$ is denoted $\widehat{A}_{\cL}$ and is defined to be the cokernel of the map $\widehat{\Lambda_r}^r\xrightarrow{\widehat{P}_{\cL}}\widehat{\Lambda_r}^r$. Let $\widehat{\Delta}_{\cL}$ be the greatest common divisor of all determinants of $(r-1)\times (r-1)$ minors of $\widehat{P}_{\cL}$. Note that $\widehat{\Delta}_{\cL}$ is well defined up to multiplication by a unit in $\widehat{\Lambda_r}$.
\par Recall that $\widehat{\Lambda}$ is the Iwasawa algebra $\widehat{\Lambda}:=\Z_p\llbracket X\rrbracket$ and let $\tau_z:\widehat{\Lambda}_r\rightarrow \widehat{\Lambda}$ be the reducing homomorphism defined by $\tau(X_i)=(1+X)^{z_i}-1$. Also set $\tau:=\tau_{\mathbf{1}}$. The reduced completed Alexander module is defined by $\widehat{A}_{\cL}^{\op{red}}:=\widehat{A}_{\cL}\hat{\otimes}_{\widehat{\Lambda}_r} \widehat{\Lambda}$, which is presented by $\tau(\widehat{P}_{\cL})$, and set
\[\tau(\widehat{\Delta}_{\cL}):=\widehat{\Delta}_{\cL}(X,X,\dots, X).\] More generally, set $\widehat{A}_{\cL}^{z}:=\widehat{A}_{\cL}\hat{\otimes}_{\widehat{\Lambda}_r, \tau_z} \widehat{\Lambda}$
and 
\[\widehat{\Delta}_{\cL,z}:=\tau_z(\widehat{\Delta}_{\cL})=\widehat{\Delta}_{\cL}\left((1+X)^{z_1}-1,(1+X)^{z_2}-1,\dots, (1+X)^{z_r}-1\right).\]
The relationship between the completed Alexander polynomial $\widehat{\Delta}_{\cL,z}\in \Z_p\llbracket X\rrbracket$ and the classical Alexander polynomial ${\Delta}_{\cL,z}\in \Z_p[t^{\pm 1}]$ is given by 
\[\widehat{\Delta}_{\cL,z}(X)= {\Delta}_{\cL,z}(1+X),\] where we plug $(1+X)$ in place of $t$ (see \cite[(3.1.6) and (3.1.7)]{hillman2005pro}). Note that this equality is up to a unit in $\widehat{\Lambda}$. Assume for the moment that $\widehat{\Delta}_{\cL,z}(X)\neq 0$. By the Weierstrass preparation theorem, we may write \begin{equation}\label{WPT}\widehat{\Delta}_{\cL,z}(X)=p^{\mu} P(X)u(X),\end{equation} where $\mu\in \Z_{\geq 0}$, $P(X)$ is a monic polynomial whose non-leading coefficients are divisible by $p$, and $u(X)$ is a unit in $\Z_p$. The $\mu$-invariant associated with $\widehat{\Delta}_{\cL,z}(X)$ is the number $\mu$ that appears in \eqref{WPT}, and the $\lambda$-invariant is the degree of the distinguished polynomial $P(X)$. We shall denote these $\mu$ and $\lambda$-invariants by $\mu_{\cL,z}$ and $\lambda_{\cL, z}$ respectively. On the other hand, if $\widehat{\Delta}_{\cL,z}(X)= 0$, we set $\mu_{\cL,z}=\lambda_{\cL, z}=\infty$. It is indeed possible for the multivariable Alexander polynomial itself to be zero, for instance, for \href{http://katlas.org/wiki/L11n244}{L11n244}, which is a link with $2$ knot components.

\par Note that in \cite{hillman2005pro}, the $\mu$ and $\lambda$-invariants are associated to the \emph{Hosokawa polynomial} and not the Alexander polynomial. The only difference is that the choice of $\lambda_{L,z}$ is $(r-1)$ more than the $\lambda$-invariant of the Hosokawa polynomial. This is only a matter of convention, we state well-known results essentially due to Mayberry-Murasugi \cite{mayberry1982torsion}, Hillman-Matei-Morishita (for $z=\mathbf{1}$) \cite[Theorem 5.1.7]{hillman2005pro} and later adapted to a slightly more general case (for general $z$) by Kadokami-Mizusawa (see \cite[Theorem 2.1]{kadokami2008iwasawa}). For the rest of this section, we shall recall these results.
\par Following the notation and conventions of \cite[section 2]{kadokami2013iwasawa}, recall that the choice of $z$ gives rise to an infinite cyclic cover $\pi_z: X_z\rightarrow X$. Given $n\in \Z_{\geq 0}$ denote by $\pi_z^n: X_{z,p^n}\rightarrow X$ the unique $\Z/p^n\Z$-subcover of $\pi_z$. Let $\xi_{ z}^n:M_{z,p^n}\rightarrow S^3$ be the \emph{Fox completion} of $\pi_{ z}^n$ (see \cite[section 10.2]{morishita2011knots} for further details). Note that $\xi_{ z}^n$ is branched along $\cL$. Given a $\Z$-module $A$, set $|A|$ to denote the cardinality of $A$ if it is finite and $0$ if $A$ is infinite. 
\begin{theorem}
Let $z$ be an admissible integral vector and assume that $\widehat{\Delta}_{\cL,z}\neq 0$. Let $v=v(z)$ be the maximum of $v_p(z_i)$, where $v_p$ is the valuation normalized by $v_p(p)=1$. We have the following:
\begin{enumerate}
    \item $|H_1(M_{z,p^n};\Z)|=|H_1(M_{z,p^v};\Z)|\times \left( \prod_{\zeta^{p^n}=1,\\
    \zeta^{p^v}\neq 1,} |\Delta_{\cL,z}(\zeta)|\right)$ for all $n\geq v$.
    \item If $|H_1(M_{z,p^n};\Z)|\neq 0$, then, there is an integer $v_{\cL,z}$ such that 
    \[|H_1(M_{z,p^n};\Z)|=p^{\left(p^n\mu_{\cL, z}+n\lambda_{\cL, z}+\nu_{\cL,z}\right)}\] for $n$ large enough.
\end{enumerate}
\end{theorem}

\begin{proof}
The stated result is \cite[Theorem 2.1]{hillman2005pro}.
\end{proof}
Note that when $r=1$, the singleton tuple $z$ is forced to be equal to $1$. In the case of a knot, the Iwasawa invariants are well understood, as the following result shows.
\begin{lemma}\label{mu lambda for a knot}
Let $\cK$ be a knot. Then, the Iwasawa invariants are given by $\mu_{\cK, 1}=0$ and $\lambda_{\cK, 1}=0$.
\end{lemma}

\begin{proof}
The above result is \cite[Theorem 2.2 (i)]{kadokami2013iwasawa}, we provide details here. As is well known, $\Delta_{\cK,\mathbf{1}}(1)=\Delta_{\cK}(1)=\pm 1$. Hence, $\widehat{\Delta}_{\cK,1}(0)=\pm 1$ is a unit in $\Z_p$. Therefore, $\widehat{\Delta}_{\cK,1}(X)$ is a unit in $\widehat{\Lambda}$, and hence, $\mu_{\cK, 1}=0$ and $\lambda_{\cK, 1}=0$.
\end{proof}

\section{Linking numbers and calculation of Iwasawa Invariants}\label{s 4}
\par In this section, we introduce various methods of detecting the Iwasawa invariants introduced in the previous section, and how they are related to \emph{$p$-adic Milnor invariants} introduced in \cite{hillman2005pro}. We shall describe the precise analogy with the number field case. Let $p$ be an odd prime and $K$ be an imaginary quadratic field in which $p$ splits, and consider the cyclotomic $\Z_p$-extension $K_{\op{cyc}}/K$. The splitting of $p$ creates a trivial zero for the $p$-adic L-function which in turn shows that the characteristic element in $\Z_p\llbracket X \rrbracket$ is divisible by $X$ (see the discussion on p.1 of \cite{sands1993non}). Therefore, in the case where $p$ splits in $K$, we have that $\lambda_p(K_{\op{cyc}}/K)\geq 1$. In this setting, Gold gives a criterion that distinguishes the case when $\lambda_p(K_{\op{cyc}}/K)=1$ from the case when $\lambda_p(K_{\op{cyc}}/K)>1$. This essentially is due a formula due to Federer, Gross and Sinnott \cite{federer1980regulators} for the leading (non-zero) term of the $p$-adic L-series which is expressed in terms of the $p$-part of the class group of $K$, and the $p$-adic regulator for the $\Z_p$-extension $K_{\op{cyc}}/K$. The result follows from work of Gold \cite{gold1974nontriviality} and Sands \cite{sands1993non}.

\begin{theorem}\label{thm of sands and others}
Let $p$ be an odd prime which is split in an imaginary quadratic field $K$. Write $p$ as a product $p\cO_K=\mathfrak{p}\mathfrak{p}^*$, where $\mathfrak{p}$ and $\mathfrak{p}^*$ are the primes of $K$ above $p$. The following conditions are equivalent:
\begin{enumerate}
    \item $\lambda_p(K_{\op{cyc}}/K)> 1$,
    \item either
    \begin{enumerate}
        \item $p$ divides the class number of $K$,
        \item or, $p$ does not divide the class number of $K$. Let $N$ be a positive integer not divisible by $p$ such that $\mathfrak{p}^N=\alpha \cO_K$, a principal ideal in $K$. Then, $\alpha^{p-1}\equiv 1\mod{(\mathfrak{p}^*)^2}$.
    \end{enumerate}
\end{enumerate}
\end{theorem}

\begin{proof}
The result follows from \cite[Proposition 2.1]{sands1993non}.
\end{proof}

In \cite{horie1987note}, Horie shows that for any odd prime $p$, there are infinitely many imaginary quadratic fields $K$ in which $p$ splits such that $\lambda_p(K_{\op{cyc}}/K)>1$. On the other hand, Jochnowitz shows in \cite{jochnowitz1994p} that there are infinitely many $K$ in which $p$ splits such that $\lambda_p(K_{\op{cyc}}/K)=1$. One may formulate a more refined question guided by arithmetic statistics. Given an imaginary quadratic field $K$, and a real number $x>0$, let $D_K$ denote its discriminant and let $\mathcal{M}_{<x}^p$ be the set of all imaginary quadratic fields in which $p$ splits with $|D_K|<x$. Note that this is a finite set. Let $\mathcal{M}_{<x}'$ be the subset of $\mathcal{M}_{<x}$ consisting of imaginary quadratic fields for which $\lambda_p(K_{\op{cyc}}/K)=1$. Then, the upper (resp. lower density) of imaginary quadratic fields $K$ in which $p$ splits and $\lambda_p=1$ is taken to be $\limsup_{x\rightarrow \infty} \frac{\#\mathcal{M}_{<x}'}{\#\mathcal{M}_{<x}}$ (resp. $\liminf_{x\rightarrow \infty} \frac{\#\mathcal{M}_{<x}'}{\#\mathcal{M}_{<x}}$). If the limit $\lim_{x\rightarrow \infty} \frac{\#\mathcal{M}_{<x}'}{\#\mathcal{M}_{<x}}$ exists, then it is called the density. Let $\bar{\mathfrak{d}}_{\lambda=1}^p$, $\underline{\mathfrak{d}}_{\lambda=1}^p$ be the upper and lower density respectively, and note that if $\bar{\mathfrak{d}}_{\lambda=1}^p=\underline{\mathfrak{d}}_{\lambda=1}^p$, the density exists, and we denote it by $\mathfrak{d}_{\lambda=1}^p$.
\begin{question}
Let $p$ be a prime. With respect to notation above, what can be said about the quantities $\bar{\mathfrak{d}}_{\lambda=1}^p$ and $\underline{\mathfrak{d}}_{\lambda=1}^p$? Are they equal, and if so, what is $\mathfrak{d}_{\lambda=1}^p$?
\end{question}
\par Such questions seem to be difficult in the number field setting and are currently being pursued by the second named author. We note here that this is a very natural question, if fact, one can ask the same question for any choice of $\lambda\geq 1$.
\par In this paper, we shall study topological analogs of the above question in section \ref{s 5}. In this section, we shall prove a number of analogs for Theorem \ref{thm of sands and others} in the topological setting.
\par Let $\cL$ be a link with $r$ component knots $\cK_1, \dots, \cK_r$. Note that in the case when $r=1$, the $\mu$ and $\lambda$-invariants are known to vanish by Lemma \ref{mu lambda for a knot}, hence we study the case when $r\geq 2$. Recall that the Alexander module is presented by the $r\times r$ matrix $\widehat{P}_{\cL}$. Let $\widehat{F}$ be the free pro-$p$ group on $x_1,\dots, x_r$, where we recall that $x_i$ corresponds to the meridian for $\cK_i$, and let $y_i$ be the $i$-th longitude. Let $I=(i_1, \dots, i_n)$ be a multi-index, the \emph{$p$-adic Milnor number} is 
$\hat{\mu}(I):=\epsilon\left(\partial_{I'}(y_{i_n})\right)$, where $I'=(i_1,\dots, i_{n-1})$. Furthermore, when $n=1$, we set $\hat{\mu}(I)=0$. Given $\ell\geq 1$, and indices $i\neq j$, we have that 
\begin{equation}\label{mu and linking number}
    \hat{\mu}(\underbrace{i\cdots i}_{\ell} j)={\ell_{i,j}\choose \ell},
\end{equation}where we recall the $\ell_{i,j}$ is the linking number for the knots $\cK_i$ and $\cK_j$. In particular, $\hat{\mu}(ij)=\ell_{i,j}$. These Milnor invariants may also be defined in terms of Massey products in the cohomology of $\widehat{G}_{\cL}$, see \cite[p. 127]{hillman2005pro} for further details. 

\begin{definition}
The $p$-adic Traldi matrix $\widehat{T}_{\cL}$ over $\widehat{\Lambda}_r$ is defined by
\[\widehat{T}_{\cL}(i,j):=\begin{cases}
-\sum_{n\geq 1} \left(\sum_{1\leq i_1,\dots, i_n\leq r, i_n\neq i} \hat{\mu}(i_1\dots i_n i) X_{i_1}\dots X_{i_n}\right)&\text{ if }i=j,\\
\hat{\mu}(ij)X_i+\sum_{n\geq 1} \left(\sum_{1\leq i_1,\dots, i_n\leq r}\hat{\mu}(i_1\dots i_n ji)X_iX_{i_1}\cdots X_{i_n}\right)
&\text{ if }i\neq j.
\end{cases}\]
\end{definition}

\begin{theorem}
The \emph{$p$-adic Traldi matrix} $\widehat{T}_{\cL}$ coincides with the Alexander matrix $\widehat{P}_{\cL}$.
\end{theorem}
\begin{proof}
The above result is \cite[Theorem 3.2.2]{hillman2005pro}.
\end{proof}
Let $z=(z_1,\dots, z_r)$ be such that $\prod_{i} z_i\neq 0$ and $\op{gcd}(z_1,\dots, z_r)=1$. Set $\widehat{T}_{\cL,z}$ to be the matrix with entries
\[\widehat{T}_{\cL,z}(i,j)=\tau_z\left(\widehat{T}_{\cL}(i,j)\right).\]
In particular, $\widehat{T}_{\cL,\mathbf{1}}$ is given by
\[\widehat{T}_{\cL,\mathbf{1}}(i,j):=\begin{cases}
-\sum_{n\geq 1} \left(\sum_{1\leq i_1,\dots, i_n\leq r, i_n\neq i} \hat{\mu}(i_1\dots i_n i) \right) X^n&\text{ if }i=j,\\
\hat{\mu}(ij)X+\sum_{n\geq 1} \left(\sum_{1\leq i_1,\dots, i_n\leq r}\hat{\mu}(i_1\dots i_n ji)\right)X^{n+1}
&\text{ if }i\neq j.
\end{cases}\]
Given $k\geq 2$, the truncated Traldi matrix $\widehat{T}_{\cL,z}^{(k)}$ is simply the mod $X^k$ reduction of $\widehat{T}_{\cL,z}$. We find that 
\[\widehat{T}_{\cL,z}^{(2)}:=\begin{cases}
\left(\sum_{t\neq i} -z_t\ell_{t,i} \right) X &\text{ if }i=j,\\
\ell_{i,j}z_iX
&\text{ if }i\neq j.
\end{cases}.\]
Here, we have used the relation $\hat{\mu}(ij)=\ell_{i,j}$, see \eqref{mu and linking number}. The \emph{linking matrix} $C_{\cL,z}$ is defined by 
\[C_{\cL,z}(i,j):=\begin{cases}
\left(\sum_{t\neq i} -z_t\ell_{t,i} \right) &\text{ if }i=j,\\
\ell_{i,j}z_i
&\text{ if }i\neq j,
\end{cases}\]
we find that $\widehat{T}_{\cL,z}^{(2)}$ is represented by $X C_{\cL,z}\mod{X^2}$.
\par Given $i, j$ such that $1\leq i,j\leq r$, let $M_{i,j}$ denote the $(r-1)\times (r-1)$ minor of $C_{\cL,z}$ obtained by deleting the $i$-th row and $j$-th column. Fix the absolute value $|\cdot |_p$ on $\Q_p$ normalized by $|p|_p=p^{-1}$ and set $|0|_p:=0$. Set $c_{\cL, z}$ to be the maximum of the values $|\op{det} M_{i,j}|_p$ as $i,j$ range over all values $1\leq i,j\leq r$. Note that $\op{rank}_{\Q_p} C_{\cL,z}<r-1$ if and only if all determinants $\op{det}M_{i,j}$ vanish, if and only if $c_{\cL,z}=0$. 

\begin{proposition}\label{prop on leading term}
Let $z=(z_1,\dots, z_r)$ be an admissible integral and $\cL$ a link with exactly $r$ components $\cK_1,\dots, \cK_r$. Then, the order of vanishing satisfies the lower bound
\[\op{ord}_{X=0} \widehat{\Delta}_{\cL, z}(X)\geq r-1.\] Write $\widehat{\Delta}_{\cL, z}(X)$ as a sum 
\[\widehat{\Delta}_{\cL, z}(X)=a_{r-1} X^{r-1}+a_r X^r+\dots, \]
where $a_{r-1}$ is the \emph{leading coefficient}. The following assertions hold:
\begin{enumerate}
    \item $\op{ord}_{X=0} \widehat{\Delta}_{\cL, z}(X)> r-1$ if and only if $\op{rank}_{\Q_p} C_{\cL,z}<r-1$.
    \item Suppose that $\op{rank}_{\Q_p} C_{\cL,z}=r-1$, then, $|a_{r-1}|_p= c_{\cL,z}$.
\end{enumerate}
\end{proposition}
\begin{proof}
Note that each entry of the matrix $\widehat{T}_{\cL,z}$ is divisible by $X$. Since $\widehat{\Delta}_{\cL,z}(X)$ is generated by the determinants of the $(r-1)\times (r-1)$ minors of $\widehat{T}_{\cL,z}$, it follows that $\op{ord}_{X=0} \widehat{\Delta}_{\cL, z}(X)\geq  r-1$. From the relation 
\[\widehat{T}_{\cL,z}\equiv C_{\cL,z}\mod{X^2}\]we thus find that $|a_{r-1}|_p$ is equal to $c_{\cL,z}$. In particular, $\op{ord}_{X=0} \widehat{\Delta}_{\cL, z}(X)> r-1$ if and only if $c_{\cL,z}=0$, if and only if $\op{rank}_{\Q_p} C_{\cL,z}<r-1$.
\end{proof}
\begin{lemma}\label{basic lemma on mu and lambda}
Let $f(X)\in \widehat{\Lambda}$ be given by the power series expansion 
\[f(X)=a_d X^d+a_{d+1} X^{d+1}+\dots, \] where $a_d\neq 0$, and let $\mu$ and $\lambda$ denote the $\mu$ and $\lambda$-invariants of $f(X)$ respectively. Then, $\lambda\geq d$ and following conditions are equivalent
\begin{enumerate}
    \item the $\mu=0$ and the $\lambda=d$,
    \item $a_d$ is a unit in $\Z_p$.
\end{enumerate}\end{lemma}
\begin{proof}
Write $f(X)=X^d g(X)$ with $g(X):=a_{r-1}+a_rX+a_{r+1}X^2+\dots$. Suppose that $\mu=0$ and $\lambda=d$. Then, $f(X)=P(X)u(X)$, where $P(X)$ is a distinguished polynomial of degree $d$ and $u(X)$ is a unit in $\widehat{\Lambda}$. Note that $X^d$ is distinguished and divides $P(X)$ and since the degree of $P(X)$ is equal to $d$, it follows that $P(X)=X^d$ and $g(X)=u(X)$ is a unit in $\widehat{\Lambda}$. It follows that $a_{r-1}$ is a unit in $\Z_p$. 
\par Conversely, if $a_{r-1}$ is a unit in $\Z_p$, then $g(X)$ is a unit in $\Lambda$ and hence, $\mu=0$ and $\lambda=\op{deg} X^d=d$.
\end{proof}
We apply the above result to the study of Iwasawa invariants. Let $\bar{C}_{\cL,z}$ be the mod-$p$ reduction of $C_{\cL,z}$.
\begin{corollary}\label{mod condition}
Let $\cL$ be a link of $r\geq 2$ component knots and $z$ an integral vector. Then, we have that $\lambda_{\cL,z}\geq r-1$. Furthermore, the following conditions are equivalent
\begin{enumerate}
    \item $\mu_{\cL, z}=0$ and $\lambda_{\cL,z}=r-1$,
    \item $c_{\cL,z}$ is a unit in $\Z_p$,
    \item $\op{rank}_{\F_p} \bar{C}_{\cL,z}=r-1$.
\end{enumerate}
\end{corollary}

\begin{proof}
Let $f(X)$ denote the completed Alexander polynomial $\widehat{\Delta}_{\cL, z}(X)$. Note that according to Proposition \ref{prop on leading term}, $\op{ord}_{X=0} f(X)\geq r-1$. In particular, this implies that $\lambda_{\cL,z}\geq r-1$. We write $f(X)=a_{r-1} X^{r-1}+a_r X^r+\dots$, and it follows from Proposition \ref{prop on leading term} that $|a_{r-1}|_p=c_{\cL,z}$. It follows from Lemma \ref{basic lemma on mu and lambda} that $c_{\cL,z}$ is a unit in $\Z_p$ if and only if $\mu_{\cL,z}=0$ and $\lambda_{\cL, z}=r-1$. It is clear that $c_{\cL,z}$ is a unit in $\Z_p$ if and only if $\op{rank}_{\F_p} \bar{C}_{\cL,z}=r-1$.
\end{proof}
Let us explain the above result for small values of $r$.
\begin{corollary}\label{cor linking number}
Let $\cL=\cK_1\cup \cK_2$ be a link consisting of $2$-components, and $z=(z_1,z_2)$ an admissible integral vector. Then, $\lambda_{\cL,z}\geq 1$ and the following conditions are equivalent
\begin{enumerate}
    \item $\ord_{X=0} \widehat{\Delta}_{\cL,z}(X)=1$,
    \item $\ell_{1,2}\neq 0$.
\end{enumerate}
Furthermore, the following conditions are equivalent
\begin{enumerate}
    \item $\mu_{\cL, z}=0$ and $\lambda_{\cL, z}= 1$.
    \item $p\nmid \ell_{1,2}$
\end{enumerate}
\end{corollary}
\begin{proof}
We have that $C_{\cL,z}=\mtx{-z_2\ell_{1,2}}{z_1\ell_{1,2}}{z_2\ell_{1,2}}{-z_1\ell_{1,2}}$. According to Proposition \ref{prop on leading term}, \[\op{ord}_{X=0} \widehat{\Delta}_{\cL, z}(X)= 1\] if and only if $\op{rank}_{\Q_p} C_{\cL,z}=1$. This is equivalent to the non-vanishing of $\ell_{1,2}$.
\par Since it is assumed that $\op{gcd}(z_1,z_2)=1$, it follows that either $z_1$ or $z_2$ (or both) are not divisible by $p$. Therefore, $\op{rank}_{\F_p} \bar{C}_{\cL,z}=1$ if an only if $p\nmid \ell_{1,2}$. 
\end{proof}

\begin{corollary}\label{3 component link corollary}
Let $\cL=\cK_1\cup \cK_2\cup \cK_3$ be a link consisting of $3$-components, and $z$ an admissible integral vector. Then, $\lambda_{\cL,z}\geq 2$ and the following conditions are equivalent
\begin{enumerate}
    \item $\op{ord}_{X=0} \widehat{\Delta}_{\cL,z}(X)=2$,
    \item  $(z_2 \ell_{2,1} \ell_{3,2}+ z_1\ell_{3,1} \ell_{1,2}+ z_3\ell_{3,1}\ell_{3,2})\neq 0$, or, $z_2 \ell_{2,1} \ell_{3,2}+ z_1\ell_{3,1} \ell_{1,2}+ z_3\ell_{3,1}\ell_{3,2}\neq 0$  (or both).
\end{enumerate}

Furthermore,  the following conditions are equivalent:
\begin{enumerate}
    \item $\mu_{\cL, z}=0$ and $\lambda_{\cL, z}= 2$.
    \item Both of the following conditions are satisfied
    \begin{enumerate}
        \item $p\nmid z_1z_2z_3$,
        \item $p\nmid \left(z_2 \ell_{2,1} \ell_{3,2}+ z_1\ell_{3,1} \ell_{1,2}+ z_3\ell_{3,1}\ell_{3,2}\right) $, or, $p\nmid\left(z_2 \ell_{2,1} \ell_{3,2}+ z_1\ell_{3,1} \ell_{1,2}+ z_3\ell_{3,1}\ell_{3,2}\right)$  (or both).
    \end{enumerate}
\end{enumerate}
\end{corollary}
\begin{proof}
We have that $C_{\cL,z}= \left( {\begin{array}{ccc}
 -(z_2 \ell_{2,1}+z_3 \ell_{3,1}) & z_1 \ell_{1,2} & z_1 \ell_{1,3}\\
 z_2 \ell_{2,1} &  -(z_1 \ell_{1,2}+z_3 \ell_{3,2})& z_2 \ell_{2,3} \\
 z_3 \ell_{3,1} & z_3 \ell_{3,2} & -(z_1 \ell_{1,3}+z_2 \ell_{2,3})\\
 \end{array} } \right)$. Proposition \ref{prop on leading term} asserts that $\op{ord}_{X=0} \widehat{\Delta}_{\cL,z}(X)=2$ if and only if $\op{rank}_{\Q_p} C_{\cL,z}=2$. Note that the sum of the rows of $C_{\cL,z}$ is equal to $0$, hence, $\op{rank}_{\Q_p}(C_{\cL,z})=2$ if and only if the first two rows are linearly independent. This is the case when at least one of the determinants
 \[\begin{split}
     & \op{det}\mtx{ -(z_2 \ell_{2,1}+z_3 \ell_{3,1})}{z_1 \ell_{1,2}}{z_2 \ell_{2,1}}{-(z_1 \ell_{1,2}+z_3 \ell_{3,2})}\\= & z_3\left(z_2 \ell_{2,1} \ell_{3,2}+ z_1\ell_{3,1} \ell_{1,2}+ z_3\ell_{3,1}\ell_{3,2}\right), \text{ and }\\
     & \op{det}\mtx{z_1 \ell_{1,2}}{z_1 \ell_{1,3}}{-(z_1 \ell_{1,2}+z_3 \ell_{3,2})}{z_2 \ell_{2,3}}\\
     =& z_1 \left(z_2\ell_{1,2} \ell_{2,3}+z_1 \ell_{1,3} \ell_{1,2}+z_3\ell_{1,3} \ell_{3,2}\right).
 \end{split}\] Note that $z_1z_2z_3\neq 0$ and thus, $\op{rank}_{\Q_p}(C_{\cL,z})=2$ if and only if $(z_2 \ell_{2,1} \ell_{3,2}+ z_1\ell_{3,1} \ell_{1,2}+ z_3\ell_{3,1}\ell_{3,2})\neq 0$, or, $z_2 \ell_{2,1} \ell_{3,2}+ z_1\ell_{3,1} \ell_{1,2}+ z_3\ell_{3,1}\ell_{3,2}\neq 0$  (or both).
 \par Corollary \ref{mod condition} asserts that $\mu_{\cL,z}=0$ and $\lambda_{\cL,z}=2$ if and only if $\op{rank}_{\F_p}\bar{C}_{\cL,z}$. This is equivalent to the above determinants not being divisible by $p$. The result follows from this.
\end{proof}

\section{Statistics for Iwasawa invariants of $2$-bridge links}\label{s 5}

\par In this section, we study the average behavior of the Iwasawa invariants associated to links $\cL$ consisting of two component knots $\cK_1$ and $\cK_2$. We formulate a precise problem in arithmetic statistics. The links $\cL$ that we consider are presented as \emph{$2$-bridge links} (see \cite[p.18]{kawauchi1996survey}).

\par To each oriented $2$-bridge link $\cL$ there corresponds a relatively prime pair of integers $(a,b)$ with $0<a<2b$, and $b$ odd when $\cL$ is a knot and $b$ even when $\cL$ is a link with two knot components $\cK_1$ and $\cK_2$. We shall denote the link associated to the pair $(a,b)$ by $\cL_{a/b}$. This construction is standard parametrization, and referred to as the \emph{Schubert normal form} and we refer to \cite[p.542, l.-6 to l.-1]{tuler1981linking} or \cite[chapter 2]{kawauchi1996survey} for an explicit description of these links. Throughout, we shall assume that $b$ is even, thus $\cL_{a/b}$ has two components. Since $a$ is assumed to be coprime to $b$, it follows that $a$ must be odd. Choose an orientation on $S^3$. All links we shall consider shall also be provided with an orientation.
\begin{definition}\cite[Definition 0.3.1]{kawauchi1996survey}
Two (oriented) piecewise linear links $\cL$ and $\cL'$ in $S^3$ have the same type if there is a piecewise linear orientation preserving homeomorphism $h:S^3\rightarrow S^3$ such that $h(\cL)=\cL'$ and $h_{|L}:\cL\rightarrow \cL'$ is orientation preserving.
\end{definition}

We shall consider equivalence classes of links $\cL_{a/b}$ of the same type. The following result explain when two pairs $\cL_{a/b}$ and $\cL_{a'/b'}$ have the same type. Note that we are assuming here that $b$ and $b'$ are both even.
\begin{theorem}
The $2$-component links $\cL_{a/b}$ and $\cL_{a'/b'}$ belong to the same type if and only if 
\[b=b'\text{ and } a'a \equiv 1\mod{2b}.\]
\end{theorem}
\begin{proof}
We refer to \cite[Theorem 2.1.3]{kawauchi1996survey} for this result.
\end{proof}
For example, it follows from the above that the oriented links $\cL_{3/10}$ and $\cL_{7/10}$ have the same type.
\begin{remark} 
The conventions in \cite{kawauchi1996survey} involve a pair $(\alpha,\beta)$ that are coprime with $-\alpha <\beta <\alpha$, and $\alpha>0$. The conventions used here are equivalent to those of \cite{kawauchi1996survey}, however, shall be of more convenience to work with. In order to see that the conventions indeed match up, compare the construction of 2-bridge links in \cite{kawauchi1996survey} with that in \cite{tuler1981linking}. We are using the convention from the latter reference.
\end{remark}
Given a rational number $a/b$ with $(a,b)=1$, and $a,b>0$, define the counting function $\op{ht}(\cL_{a/b}):=\op{max}\{a^2,b\}$, which we shall refer to as the \emph{height} of $\cL_{a/b}$. We do not work with the conventional notion of height of the fraction $a/b$, and as far as we are aware, there is no conventional notion of height that is used in such a context. In fact, our computations have led us to believe that the result is not true for an alternate definition of height $\op{ht}'(a/b):=\op{max}\{a,b\}$. This proves to be an adequate counting function to serve the purpose of formulating and resolving statistical problem for Iwasawa invariants. We abbreviate $\op{ht}(\cL_{a/b})$ to $\op{ht}(a/b)$ in order to simplify notation.

\par We shall study the average behavior of Iwasawa invariants of $\cL_{a/b}$, where $\op{ht}(a/b)$ goes to infinity. Let $\mathcal{N}$ be the set of all fractions $a/b$ with $(a,b)=1$, $b$ even, $0<a<2b$. For $x>0$, set $\mathcal{N}_{<x}$ to be the subset of $\mathcal{N}$ consisting of fractions $a/b$ such that $\op{ht}(a/b)<x$. It is easy to see that $\mathcal{N}_{<x}$ is finite, we set $\#\mathcal{N}_{<x}$ to denote its cardinality. Given a subset $\mathcal{S}$ of $\mathcal{N}$, let $\mathcal{S}_{<x}:=\mathcal{S}\cap \mathcal{N}_{<x}$. We say that $\mathcal{S}$ has density $\mathfrak{d}(\mathcal{S})$ if the following limit exists
\[\mathfrak{d}(\mathcal{S}):=\lim_{x\rightarrow \infty} \frac{\# \mathcal{S}_{<x}}{\# \mathcal{N}_{<x}}.\] Denote by $\mathcal{S}^{p,z}$ the set of all fractions $a/b$ in $\mathcal{N}$ such that $\mu_{p,\cL_{a/b},z}=0$ and $\lambda_{p,\cL_{a/b},z}=1$. Our main result is as follows.
\begin{theorem}\label{main th}
Let $p$ be a fixed odd prime. With respect to notation above, the set $\mathcal{S}^{p,z}$ has density given by $\mathfrak{d}(\mathcal{S}^{p,z})=1-1/p$.
\end{theorem}
\begin{remark}
We thus study the proportion of $\cL_{a/b}$ with $2$-components and Iwasawa invariants $\mu=0$ and $\lambda=1$. Note that for simplicity we do not count links up to equivalence, and thus the equivalence class of $\cL_{a,b}$ is counted twice if $a^2\not \equiv 1\mod{2b}$ and only once if $a^2\equiv 1\mod{2b}$.
\end{remark}
In order to ease notation, we set $\mu_{p,a/b,z}:=\mu_{p,\cL_{a/b}, z}$ and $\lambda_{p,a/b,z}:=\lambda_{p,\cL_{a/b},z}$. Write $\cL=\cK_1\cup \cK_2$, and recall that $\ell_{1,2}$ is the linking number of $\cK_1$ and $\cK_2$. We set $\ell_{a/b}$ to denote the linking number $\ell_{1,2}$ for the two component knots in $\cL_{a/b}$. Note that this only makes sense since $b$ is assumed to be even. Corollary \ref{cor linking number} asserts that  $\lambda_{p,a/b,z}\geq 1$ and the following conditions are equivalent:
\begin{enumerate}
    \item $\mu_{p,a/b,z}=0$ and $\lambda_{p,a/b,z}=1$,
    \item $p\nmid \ell_{a/b}$.
\end{enumerate}

Note that the second condition is independent of $z$, and hence, so is the first condition. We find that $a/b$ belongs to $\mathcal{S}^{p,z}$ if and only if $p\nmid \ell_{a/b}$, and this condition is independent of $z$. Also observe that if $\lambda_{p,a/b,z}=1$, then in particular, we have that $\widehat{\Delta}_{\cL,z}\neq 0$ and $\op{ord}_{X=0}\widehat{\Delta}_{\cL,z}=1$. Before giving the proof of Theorem \ref{main th}, we prove a few preparatory results.
\par For $i\in \Z$, set $\epsilon_i:=(-1)^{\lfloor \frac{ia}{b}\rfloor}$. It is easy to see from the construction of $\cL_{a/b}$ that $\ell_{a/b}$ is given by 
\begin{equation}\label{linking number formula}\ell_{a/b}=\sum_{i=1}^{b/2} \epsilon_{2i-1},\end{equation}
we refer to \cite{tuler1981linking} for further details. 

\begin{lemma}
The following relation is satisfied 
\[\ell_{a/(b+2a)}=\ell_{a/b}+1.\]
\end{lemma}
\begin{proof}
For $0\leq j \leq (a-1) $, let $c_j(a/b)$ be the number of odd integers $2i-1$ such that \[\frac{jb}{a}\leq 2i-1<\frac{(j+1)b}{a}.\] Note that $2i-1$ is in the range $\frac{jb}{a}\leq 2i-1< \frac{(j+1)b}{a}$, if and only if $\lfloor \frac{(2i-1)a}{b}\rfloor =j $. Therefore, $c_j(a/b)$ is the number of odd numbers $2i-1$ such that $\lfloor \frac{(2i-1)a}{b} \rfloor=j$. Thus the formula \eqref{linking number formula} yields
\begin{equation}\label{new linking number formula}\ell_{a/b}=\sum_{j=0}^{a-1} (-1)^j c_j(a/b).\end{equation}
Next, we show that $c_j(\frac{a}{b+2a})=c_j(\frac{a}{b})+1$. Note that $c_j(\frac{a}{b+2a})$ is the number of odd integers $2i-1$ such that 
\[\frac{jb}{a}+2j\leq 2i-1<\frac{(j+1)b}{a}+2j+2.\] Let $\{2i_1-1,\dots, 2i_{c_j}-1\}$ be the odd numbers such that 
\[\frac{jb}{a}\leq 2i_k-1<\frac{(j+1)b}{a},\] here, $c_j=c_j(a/b)$.
Then, $\{2i_1+2j-1,\dots, 2i_{c_j}+2j-1\}$ are the odd numbers $2i-1$ such that 
\[\frac{jb}{a}+2j\leq 2i-1<\frac{(j+1)b}{a}+2j.\] There is exactly one odd number $2i-1$ in the range
\[\frac{(j+1)b}{a}+2j\leq 2i-1<\frac{(j+1)b}{a}+2j+2,\] namely, $2i-1=2i_{c_j}+2j+1$. Hence, this shows that $c_j(\frac{a}{b+2a})=c_j(\frac{a}{b})+1$.
It follows from \eqref{new linking number formula} that 
\[\begin{split}\ell_{a/(b+2a)}=& \sum_{j=0}^{a-1} (-1)^j c_j\left(a/(b+2a)\right)\\
=& \sum_{j=0}^{a-1} (-1)^j \left(c_j(a/b)+1\right)\\
=& \ell_{a/b}+\sum_{j=0}^{a-1} (-1)^j \\
=& \ell_{a/b}+1. \\
\end{split}\]
The last equality holds since $a$ is odd.
\end{proof}

Since the set $\mathcal{S}^{p,z}$ consists of all fractions $a/b\in \mathcal{N}$ with $p\nmid \ell_{a/b}$, we find that $\mathcal{S}^{p,z}$ does not depend on the choice of $z$. Thus, to simplify notation, we denote this set by $\mathcal{S}^p$. Given a fixed number $a>0$, set \[\begin{split}&\mathcal{N}_{<x}(a):=\{b\mid 2b>a, a/b\in \mathcal{N}_{<x}\}.\\
&\mathcal{S}_{<x}^{p}(a):=\{b\mid 2b>a, a/b\in \mathcal{S}_{<x}^p\}.\\
\end{split}\] 
\begin{lemma}\label{ad hoc lemma 1}
Let $a>0$, with respect to notation above 
\[|\frac{1}{p} \# \mathcal{N}_{<x}(a)-\# \mathcal{S}_{<x}^p(a)|<2a.\]

\end{lemma}

\begin{proof}
The relation $\ell_{a/(b+2a)}=\ell_{a/b}+1$ shows that among every $2ap$ consecutive numbers $b$ in $\mathcal{N}_{<x}(a)$, there are exactly $2a$ of them that belong to $\mathcal{S}^p_{<x}(a)$. Therefore, we find that 
\[| \# \mathcal{N}_{<x}(a)-p\# \mathcal{S}^p_{<x}(a)|<2ap.\] The result follows from this.
\end{proof}

\begin{lemma}
With respect to notation above, 
\[|\frac{1}{p} \# \mathcal{N}_{<x}-\# \mathcal{S}_{<x}^p|<x+\sqrt{x}.\]
\end{lemma}
\begin{proof}
Note that $\op{ht}(a/b)=\op{max}\{a^2,b\}$, thus if $\op{ht}(a/b)<x$, we have that $a<\lfloor \sqrt{x}\rfloor$.
We find that 
\[\begin{split}
    & |\frac{1}{p} \# \mathcal{N}_{<x}-\# \mathcal{S}_{<x}^p|\\
    \leq  & \sum_{a=1}^{\lfloor \sqrt{x}\rfloor}|\frac{1}{p} \# \mathcal{N}_{<x}(a)-\# \mathcal{S}_{<x}^p(a)|\\
    <&\sum_{a=1}^{\lfloor \sqrt{x}\rfloor}2a=\lfloor \sqrt{x}\rfloor (\lfloor \sqrt{x}\rfloor+1).\\
\end{split}\]
Note that in the last step, we have invoked Lemma \ref{ad hoc lemma 1}. The result follows.
\end{proof}

\begin{lemma}\label{ad hoc lemma 2}
With respect to notation above, 
\[\#\mathcal{N}_{<x}\geq x^{3/2}-\frac{5}{4} (x+\sqrt{x}).\]
\end{lemma}
\begin{proof}
We find that 
\[\#\mathcal{N}_{<x}=\sum_{a=1}^{\lfloor \sqrt{x}\rfloor} \#\mathcal{N}_{<x}(a).\]
Recall that $\mathcal{N}_{<x}(a)$ is the set of numbers $b$ such that $2b>a$ and $a/b\in \mathcal{N}_{<x}$. Since $a\leq \lfloor \sqrt{x}\rfloor$, the condition $a/b\in \mathcal{N}_{<x}$ simply becomes $b\leq \lfloor x\rfloor$. Therefore, we have that $\#\mathcal{N}_{<x}(a)\geq x-a/2-1$, and thus,
\[\begin{split} \#\mathcal{N}_{<x}\geq  &\sum_{a=1}^{\lfloor \sqrt{x}\rfloor} (x-a/2-1).\\
= & \lfloor \sqrt{x}\rfloor x-\frac{1}{4} \lfloor \sqrt{x}\rfloor(\lfloor \sqrt{x}\rfloor+1)-\lfloor \sqrt{x}\rfloor\\
> & x^{3/2}-x-\frac{1}{4} \sqrt{x}( \sqrt{x}+1)- \sqrt{x}\\
=& x^{3/2}-\frac{5}{4} (x+\sqrt{x}). \end{split}\]
\end{proof}

\begin{proof}[Proof of Theorem \ref{main th}]
Recall that Lemma \ref{ad hoc lemma 1} asserts that
\[\left|\frac{1}{p} \# \mathcal{N}_{<x}-\# \mathcal{S}_{<x}^p\right|<x+\sqrt{x},\]
and therefore, 
\[\left|\frac{1}{p} -\frac{\# \mathcal{S}_{<x}^p}{\# \mathcal{N}_{<x}}\right|<\frac{x+\sqrt{x}}{\# \mathcal{N}_{<x}}.\]
Invoking Lemma \ref{ad hoc lemma 2}, we find that 
\[\left|\frac{1}{p} -\frac{\# \mathcal{S}_{<x}^p}{\# \mathcal{N}_{<x}}\right|<\frac{x+\sqrt{x}}{\left(x^{3/2}-\frac{5}{4} (x+\sqrt{x})\right)},\] which clearly goes to $0$ as $x\rightarrow \infty$. The result thus follows.

\end{proof}

\begin{remark}
Note that the above proof in fact shows that the error term is bounded by $\frac{x+\sqrt{x}}{\left(x^{3/2}-\frac{5}{4} (x+\sqrt{x})\right)}$, a quantity which is asymptotic to ${x}^{-1/2}$.
\end{remark}

Therefore, we have shown that the Iwasawa invariants do satisfy $\mu_{p,a/b,z}=0$ and $\lambda_{p,a/b,z}=1$ for a density of $(1-1/p)$ of the $2$-bridge links parametrized according to $\op{ht}(\cL_{a/b})=\op{max}\{a^2,b\}$. This does not however give much information about the links for which $\mu_{p,a/b,z}>0$ or $\lambda_{p,a/b,z}>1$. We do believe that more answers are obtainable for such links, for instance, it is certainly of interest to determine if indeed $\mu=0$ holds for a density $1$ set of links parametrized in this way for the integral vector $z=\mathbf{1}=(1,1)$. Although the authors have not been able to answer this question, it seems that some more refined answers may be obtainable when $a$ is chosen to a fixed odd number, and $b$ is allowed to vary over all even integers such that $2b>a$. We illustrate this by choosing $a=1$, thus, we are interested in the Iwasawa invariants for the family of links $\cL_{1,n}$ parametrized with respect to the variable $n$. Note that $p$ is a fixed odd prime and we take $z=\mathbf{1}=(1,1)$ for simplicity. In this setting, $\Delta_{a/b}(t):=\Delta_{\cL_{a/b},\mathbf{1}}(t)$ is given by a formula due to Minkus, see \cite{minkus1982branched} or \cite[Theorem 1]{hoste2020note}
\begin{equation}\label{minkus}
\Delta_{a/b}(t)=\sum_{k=0}^{b-1}(-1)^k t^{\sum_{i=0}^k \epsilon_k},
\end{equation}
where we recall that $\epsilon_k=(-1)^{\lfloor \frac{ka}{b}\rfloor }$. Note that since for the fraction $1/n$, we have that $\epsilon_k=1$ since $k<n$ in the above sum. Therefore, we have that 
\begin{equation}\label{minkus formula}\Delta_{1/n}(t)=\sum_{k=0}^{n-1}(-1)^k t^{k}=\frac{(-t)^n-1}{-t-1}=\frac{t^n-1}{-t-1},\end{equation}
where the last equality follows since $n$ is even.
\begin{lemma}\label{boring lemma 1}
Let $f(X)=\sum_i c_i X^i\in \widehat{\Lambda}$ be a formal power series whose coefficients are not all divisible by $p$, then, the $\mu$-invariant is $0$. Furthermore, the $\lambda$-invariant is the smallest value of $i$ such that $p\nmid c_i$.
\end{lemma}
\begin{proof}
We omit the proof since it is a straightforward exercise involving the Weierstrass factorization theorem. 
\end{proof}

\begin{theorem}
Assume that $p$ is an odd prime. With respect to notation above, the following assertions hold:
\begin{enumerate}
    \item $\mu_{p,1/n, \mathbf{1}}=0$ for all $n$.
    \item Write $n=p^t m$ with $p\nmid m$, then, $\lambda_{p,1/n, \mathbf{1}}=p^t$.
\end{enumerate}
\end{theorem}
\begin{proof}
From \eqref{minkus formula}, we find that
\[\widehat{\Delta}_{1/n}(X)=\Delta_{1/n}(1+X)=-\frac{(1+X)^n-1}{(2+X)}=-\frac{\sum_{i=1}^n {n\choose i} X^i}{(2+X)}.\] Note that since $p>2$, the denominator $(2+X)$ is a unit in $\widehat{\Lambda}$. Recall that $n=p^tm$ where $m$ is comprime to $p$. On the other hand, it follows from elementary properties of binomial coefficients that $n\choose i$ is divisible by $p$ for all $i<p^t$, and $p\nmid {n\choose p^t}$. Since $p\nmid {n\choose p^t}$, it follows that $p$ does not divide all coefficients of $\widehat{\Delta}_{1/n}(X)$, hence, $\mu_{p,1/n,\mathbf{1}}=0$. Thus, from Lemma \ref{boring lemma 1}, it follows that $\lambda_{p,1/n,\mathbf{1}}$ is the minimal value $i$ such that $p\nmid {n\choose i}$. Thus, we have shown that $\lambda_{p,1/n,\mathbf{1}}=p^t$.
\end{proof}
In light of Corollary \ref{3 component link corollary}, we expect that the results proven in this section for the family of $2$-bridge links potentially generalize to the family of $3$-bridge links. We do not however pursue such questions in this paper.

\section{Evidence for $\mu=0$ for $100\%$ of $2$-bridge links}\label{s 6}
\par In this final section, we study the vanishing of the $\mu$-invariant for $2$-bridge links $\cL=\cL_{a/b}$, for the integral vector $z=\mathbf{1}=(1,1)$. Note that this vector specifies a $\Z_p$-cover of $X_{\cL}$, whose Alexander polynomial is given by \eqref{minkus formula}. There is a parallel to this on the arithmetic side, namely, the cyclotomic $\Z_p$-extension of a quadratic number field. By the celebrated theorem of Fererro and Washington \cite{ferrero1979iwasawa}, for any prime $p$, the $\mu$-invariant is known to vanish for the cyclotomic $\Z_p$-extension $F_{\op{cyc}}$ of any abelian number field extension $F/\Q$. In analogy with the number fields case, one may ostensibly predict that $\mu_{p,a/b,\mathbf{1}}=0$ for all $2$-bridge links $\mathcal{L}_{a/b}$. However, this is seen to not be the case by examining some examples. For instance, the completed Alexander polynomial for $\mathcal{L}_{5/6}$ and $z = \mathbf{1}$ is given by \[\widehat{\Delta}_{5/6}(X) = 3(1 + X)(1-X)\] by Minkus' formula \eqref{minkus}, and thus, $\mu_{3,5/6,\mathbf{1}}=1$. Our computations however indicate that the $\mu$-invariant vanishes \emph{on average}, and have led us to the following conjecture.
\begin{conjecture}
Let $p$ be a fixed odd prime. Denote by $\mathcal{U}^p$ the set of all fractions $a/b$ in $\mathcal{N}$ such that $\mu_{p,a/b,\mathbf{1}}=0$. Then, the set $\mathcal{U}^p$ has density given by $\mathfrak{d}(\mathcal{U}^p) = 1$, where
\[\mathfrak{d}(\mathcal{U}^p):=\lim_{x\rightarrow \infty} \frac{\# \mathcal{U}^p_{<x}}{\#\mathcal{N}_{<x}}.\]
\end{conjecture}
To simplify the notation, set \[\mathfrak{d}_{<x} := \frac{\# \mathcal{U}^p_{<x}}{\# \mathcal{N}_{<x}}. \]
Our computations were performed on SageMath, see \cite{stein2007sage} For our computations, we approximate $(1+X)^{-1}$ by $\sum_{i=0}^9 (-1)^i x^i$.
and we compute completed Alexander polynomials by the formula \eqref{minkus}. Our code is available at \href{https://www.anweshray.com/publications}{this link}.
\begin{table}[h]\label{only table}
\centering
\begin{tabular}{ |c||c| }
 \hline
 $p$&$\mathfrak{d}_{<1000}$\\
 \hline
 3&0.99539\\
 5&0.99820\\
 7&0.99887\\
 11&0.99955\\
 13&0.99977\\
 17&0.99977\\
 19&0.99977\\
 23&1\\
 \hline
\end{tabular}
\caption{Proportion $\mathfrak{d}_{<1000}$ of $2$-bridge links $\mathcal{L}_{a/b}$ up to $\mathrm{ht}(a/b)=1000$ such that $\mu_{p,a/b,\mathbf{1}}=0$.}
\end{table}

The numbers on the right are soon to approach $1$ as $x$ gets larger, and thus this gives us reason to believe that the limit is $1$ as $x\rightarrow \infty$. The rate of convergence seems to be faster as $p$ gets larger. 
\begin{remark}
Computations made with various other definitions of height do not yield the same results, at least it seems that the rate of convergence to $1$ is much slower for with respect to the counting functions $\op{ht}'(a/b):=\op{max}\{a,b\}$ or $\op{ht}''(a/b):=\op{max}\{a,b^2\}$. We have similar expectations for all $z=(z_1,z_2)$ such that $p\nmid z_1z_2$.
\end{remark}
\bibliographystyle{alpha}
\bibliography{references}
\end{document}